\documentclass{article}
\usepackage{amssymb}
\usepackage{amsmath}%
\usepackage{amsthm}
\newcommand{\abs}[1]{\left\vert#1\right\vert}     
\newcommand{\norm}[1]{\left\Vert#1\right\Vert}    
\newcommand{\set}[1]{\left\{\,#1\,\right\}}       
\newcommand{\paran}[1]{\left(#1\right)}           
\newcommand{\paranp}[1]{\left[#1\right]}          
\newcommand{\To}{\rightarrow}                     
\newcommand{\unu}[1]{#1^{\prime}}                 

\newcommand{\dt}{\, \mathrm{d}t}
\newcommand{\du}{\, \mathrm{d}u}

\newcommand{\dg}{\, \mathrm{d}g(t)}
\newcommand{\R}{\mathbb R}    
\newtheorem{theorem}{Theorem}
\theoremstyle{definition}

\newtheorem{lemma}[theorem]{Lemma}

\begin{document}
%
\title{The rate of convergence of some Riemann-Stieltjes sums}
\author{Adrian Holho\c s}
\maketitle
%
\begin{abstract}
We give the rate of convergence of some optimal lower Riemann-Stieltjes sums toward the integral.
\end{abstract}
%
%

\section{Introduction}
Let $[a,b]$ be a bounded closed interval. Let $f,g$ be two functions defined on $[a,b]$. Consider an  $n$-division $\Delta$ of $[a,b]$ defined by
\[\Delta\colon\ a=t_{0}<t_{1}<t_{2}<\dots <t_{n}=b\]
and consider $\xi=(\xi_1,\xi_2,\dots,\xi_n)$ such that $\xi_i\in[t_{i-1},t_{i}]$ for every $1\leq i\leq n$. The Riemann-Stieltjes sum is defined by
\[RS(f,g,\Delta,\xi)=\sum_{i=1}^{n}f(\xi_i)\cdot [g(t_{i})-g(t_{i-1})].\]
The function $f$ is said to be Riemann-Stieltjes integrable with respect to $g$ if there is an $I\in\R$ with the property that for every $\varepsilon>0$ there exists $\delta>0$ such that for every division $\Delta$ of $[a,b]$ with mesh $\norm{\Delta}=\max_{1\leq i\leq n}(t_{i}-t_{i-1})$ less than $\delta$ and every choice of the points $\xi_{i}$ in $[t_{i-1},t_{i}]$ we have $$\abs{RS(f,g,\Delta,\xi)-I}<\varepsilon.$$
The number $I$ is denoted $\int_{a}^{b}f(t)\dg$ and is called the Riemann-Stieltjes integral of $f$ with respect to $g$. When $g(x)=x$ we obtain the Riemann integrability.\par
Consider the lower Riemann-Stieltjes sum of a continuous function $f$ on $[a,b]$
\[RS(f,g,\Delta,\min)=RS(f,g,\Delta,\xi),\] where the points $\xi_{i}$ are chosen such that $f(\xi_i)=\min_{t\in[a,b]}f(t)$. The set of all $n$-divisions of $[a,b]$ is compact and $\Delta\mapsto RS(f,g,\Delta,\min)$ is continuous, so there is an optimal $n$-division $\Delta_{opt}$ at which the lower Riemann-Stieltjes sum is maximum. This optimal $n$-division may not be unique, but the sum $RS(f,g,\Delta_{opt},\min)$ is unique.\par 
In Theorem \ref{holhos-theorem-main} we give the rate of approximation of the Riemann-Stieltjes integral by the optimal lower Riemann-Stieltjes sums, a result which ge\-ne\-ra\-li\-zes Theorem 1.2 of \cite{holhos-Tas09}.
\section{Main results}
We give next a generalization of a Lemma found in \cite{holhos-Gle79}.
\begin{lemma}\label{holhos-gleason-general}
Let $g\colon[a,b]\To\R$ be a non-negative continuous function which is not identically zero on any open subinterval of $[a,b]$ and let $h:[a,b]\To\R$ be a strictly positive and continuous function on $[a,b]$. For any positive integer $n$ there exists a division of $[a,b]$:
\[t_{0}=a<t_{1}<t_{2}<\dots <t_{n-1}<t_{n}=b\] such that the quantities
\[\paran{t_{i}-t_{i-1}}\cdot \max_{t\in [t_{i-1},t_{i}]}g(t)\cdot \max_{t\in [t_{i-1},t_{i}]}h(t),\qquad 1\leq i\leq n\]
are all equal to each other. Moreover, if $J_{n}$ is the common value of all these quantities, then
\[\lim_{n\rightarrow\infty} nJ_{n}=\int_{a}^{b}g(t)h(t)\dt.\]
\end{lemma}
\begin{proof} Parametrize the $(n-1)$ simplex $\sigma$ by $n$-tuples $(u_{1},u_{2},\dots,u_{n})$, where $u_{i}\geq 0$ and $\sum_{i=1}^{n}u_{i}=1$. Let this $n$-tuple correspond to the partition of $[a,b]$ given by
\[t_{i}=a+ \paran{b-a}(u_{1}+u_{2}+\dots+u_{i}),\quad 1\leq i\leq n,\text{ and }t_{0}=a.\]
Let us define the function
\[\psi(u_{1},u_{2},\dots,u_{n})=(w_{1},w_{2},\dots,w_{n}),\quad w_{i}= \frac{v_{i}}{\sum_{i=1}^{n}v_{i}},\]where 
$v_{i}=\paran{b-a}u_{i}\cdot\max_{t\in[t_{i-1},t_{i}]}g(t)\cdot\max_{t\in[t_{i-1},t_{i}]}h(t).$\par
We have
\[\sum_{i=1}^{n}v_{i}=\sum_{i=1}^{n}(t_{i}-t_{i-1})\max_{t\in[t_{i-1},t_{i}]}g(t)\cdot\max_{t\in[t_{i-1},t_{i}]}h(t)\] is an upper Riemman sum for $\int_{a}^{b}g(t)h(t)\dt$ and $\sum_{i=1}^{n}v_{i}>0$. \par
Since the maximum value of a continuous function over a closed interval depends continuously on the endpoints of that interval, $\psi$ is a continuous function. Because $w_{i}=0$ implies $u_{i}=0$, $\psi$ maps every face of $\sigma$ into itself. All this prove that $\psi$ is surjective. So there exists $(u_{1},u_{2},\dots,u_{n})$ such that $\psi(u_{1},u_{2},\dots,u_{n})=(\frac{1}{n},\frac{1}{n},\dots,\frac{1}{n})$. This proves the first part of our Lemma.\par
We have $nJ_{n}\leq (b-a)\norm{g}\cdot\norm{h}$. Let $\varepsilon>0$ be given. From the continuity of $g$ and $h$ on $[a,b]$ there is a $\delta>0$ so that $|t-t'|<\delta$ implies $|g(t)-g(t')|<\varepsilon$ and $|h(t)-h(t')|<\varepsilon$. We choose $n> \frac{(b-a)\norm{g}\cdot\norm{h}}{\varepsilon^2\delta}$. \par
If $\max_{t\in [t_{i-1},t_{i}]}g(t)\geq \varepsilon$ and $\max_{t\in [t_{i-1},t_{i}]}h(t)\geq \varepsilon$ we have $$J_{n}=\paran{t_{i}-t_{i-1}}\cdot \max_{t\in [t_{i-1},t_{i}]}g(t)\cdot \max_{t\in [t_{i-1},t_{i}]}h(t)\geq \varepsilon^2 (t_{i}-t_{i-1}),$$ which proves that $$t_{i}-t_{i-1}\leq \frac{J_{n}}{\varepsilon^2}= \frac{nJ_{n}}{n\varepsilon^2}\leq \frac{(b-a)\norm{g}\cdot\norm{h}}{n\varepsilon^2}<\delta.$$ This implies that the oscillations of $g$ and of $h$ over $[t_{i-1},t_{i}]$ are at most $\varepsilon$. Considering $\eta_i$ and $\xi_i$ the points of maximum for $g$ and $h$ over the interval $[t_{i-1},t_{i}]$ and applying the Mean Value Theorem for integrals we obtain
\begin{flalign*}
\abs{nJ_{n}-\int_{a}^{b}g(t)h(t)\dt}&=\abs{\sum_{i=1}^{n}\paranp{g(\eta_{i})h(\xi_i)-g(c_{i})h(c_{i})}(t_{i}-t_{i-1})}\\
&\leq \sum_{i=1}^{n}\paran{|g(\eta_i)-g(c_i)|\cdot |h(\xi_i)|+ |g(c_i)|\cdot|h(\xi_i)-h(c_i)|}(t_{i}-t_{i-1})\\
&\leq \varepsilon (\norm{h}+\norm{g})(b-a).
\end{flalign*}
This proves that $nJ_{n}$ tends to $\int_{a}^{b}g(t)h(t)\dt$.
 \par
Consider now the case when $\max_{t\in [t_{i-1},t_{i}]}g(t)< \varepsilon$ or $\max_{t\in [t_{i-1},t_{i}]}h(t)< \varepsilon$. Suppose $g(t)<\varepsilon$ for every $t\in [t_{i-1},t_{i}]$. The case when $\max_{t\in[t_{i-1},t_{i}]} h(t)<\varepsilon$ can be analysed similarly. Because $g$ is nonnegative we deduce also that the oscillation of $g$ over the interval $[t_{i-1},t_{i}]$ is at most $\varepsilon$. As we have done before $nJ_{n}$ differs from the integral $\int_{a}^{b}g(t)h(t)\dt$ by less than $\varepsilon(b-a)3\norm{h}$. 
\end{proof}
\begin{lemma}\label{holhos-lemma-estimate}
For every function $f\in C^{1}[a,b]$ and every $g\in C^{1}[a,b]$ with $\unu{g}(t)>0$, for every $t\in [a,b]$, we have
\[\int_{a}^{b}f(t)\dg -[g(b)-g(a)]\min_{t\in[a,b]}f(t)\leq \frac{1}{2}(b-a)^2 \norm{\unu{f}}\cdot\norm{\unu{g}}.\] 
\end{lemma}
\begin{proof}
Let $c\in[a,b]$ be the minimum point of $f$ over $[a,b]$. We have
\begin{flalign*}
\int_{a}^{b}f(t)\dg& -[g(b)-g(a)]\min_{t\in[a,b]}f(t)=\int_{a}^{b}[f(t)-f(c)]\unu{g}(t)\dt\\
&\leq \norm{\unu{f}}\cdot\norm{\unu{g}}\cdot \int_{a}^{b}|t-c|\dt.
\end{flalign*}
The proof is completed by using the inequality:
\[\int_{a}^{b}|t-c|\dt= \frac{(c-a)^2}{2}+ \frac{(b-c)^2}{2}\leq \frac{(b-a)^2}{2}.\]
\end{proof}
\begin{lemma}\label{holhos-lemma-main-1}
Consider $f$ a function of class $C^1$ defined on $[a,b]$ with the derivative $\unu{f}$ having a finite number of zeros. Let $g\in C^{1}[a,b]$ be a function with $\unu{g}(t)>0$ for every $t$ in $[a,b]$. Then
\[\limsup_{n\rightarrow\infty}\ n\paran{\int_{a}^{b}f(t)\dg-RS(f,g,\Delta_{opt},\min)}\leq\frac{1}{2}\paran{\int_{a}^{b}\sqrt{|\unu{f}(t)|\cdot \unu{g}(t)}\dt}^2.\]
\end{lemma}
\begin{proof}
We apply Lemma \ref{holhos-gleason-general} to the functions $\abs{\unu{f}(t)}^{\frac{1}{2}}$ and $\abs{\unu{g}(t)}^{\frac{1}{2}}$ and obtain a division $\Delta':\ a=t_{0}<t_{1}<t_{2}<\dots <t_{n}=b$ such that
\[J_{n}=\paran{t_{i}-t_{i-1}}\cdot \max_{t\in [t_{i-1},t_{i}]}\abs{\unu{f}(t)}^{\frac{1}{2}}\cdot \max_{t\in [t_{i-1},t_{i}]}\abs{\unu{g}(t)}^{\frac{1}{2}},\] has the same value for all values of $i\in\set{1,2,\dots,n}$ and 
\[\lim_{n\rightarrow\infty} nJ_{n}=\int_{a}^{b}\sqrt{|\unu{f}(t)|\cdot \unu{g}(t)}\dt.\]
Using Lemma \ref{holhos-lemma-estimate} we obtain
\begin{flalign*}
\int_{a}^{b}f(t)\dg-RS(f,g,\Delta',\min)&=\sum_{i=1}^{n}\paran{\int_{t_{i-1}}^{t_i}f(t)\dg-[g(t_i)-g(t_{i-1})]\min_{t\in[t_{i-1},t_i]}f(t)}\\& \leq \frac{1}{2}\sum_{i=1}^{n}(t_{i}-t_{i-1})^2\cdot \max_{t\in [t_{i-1},t_i]} |\unu{f}(t)|\cdot \max_{t\in [t_{i-1},t_i]} |\unu{g}(t)|
\end{flalign*}
and finally
\begin{flalign*}
\limsup_{n\rightarrow\infty}\ n&\paran{\int_{a}^{b}f(t)\dg-RS(f,g,\Delta_{opt},\min)}\\&\leq \limsup_{n\rightarrow\infty}\ n\paran{\int_{a}^{b}f(t)\dg-RS(f,g,\Delta',\min)}\\&\leq \limsup_{n\rightarrow\infty} n\frac{1}{2}nJ_{n}^2= \frac{1}{2}\lim_{n\rightarrow\infty} (nJ_{n})^2\\&= \frac{1}{2}\paran{\int_{a}^{b}\sqrt{|\unu{f}(t)|\cdot \unu{g}(t)}\dt}^2.
\end{flalign*}
\end{proof}
\begin{lemma}\label{holhos-lemma-modulus}
Consider $f$ a function of class $C^1$ defined on $[a,b]$. Let $g\in C^{1}[a,b]$ be a function with $\unu{g}(t)>0$ for every $t$ in $[a,b]$. If $\unu{f}(t)\neq 0$ in a subinterval $[p,q]$ of $[a,b]$, then for every $\xi\in[p,q]$ we have
\begin{flalign*}&\abs{\int_{p}^{q}f(t)\dg-[g(q)-g(p)]\min_{t\in[p,q]}f(t)- \frac{1}{2}(q-p)^2|\unu{f}(\xi)|\unu{g}(\xi)}\\&\leq \frac{1}{2}(q-p)^2\cdot \paranp{\norm{\unu{g}}\cdot\omega(\unu{f},q-p)+\norm{\unu{f}}\cdot\omega(\unu{g},q-p)}\end{flalign*} where $\omega(h,\delta)$ is the usual modulus of continuity of the function $h$.
\end{lemma}
\begin{proof}
Suppose $\unu{f}>0$ on $[p,q]$. The case when the derivative of $f$ is strictly negative on $[p,q]$ can be treated similarly. Because $f$ is strictly increasing the minimum of $f$ is attained in $p$. We have
\begin{flalign*}
\int_{p}^{q}f(t)\dg-[g(q)-g(p)]\min_{t\in[p,q]}f(t)=\int_{p}^{q}[f(t)-f(p)]\unu{g}(t)\dt.
\end{flalign*}
Applying the Mean Value Theorem for integrals twice we obtain
\begin{flalign*}
\int_{p}^{q}[f(t)-f(p)]\unu{g}(t)\dt&=\unu{g}(c)\cdot \int_{p}^{q}[f(t)-f(p)]\dt\\&=\unu{g}(c)\cdot \int_{p}^{q}\int_{p}^{t}\unu{f}(u)\du\dt\\&=\unu{g}(c)\cdot \int_{p}^{q}\unu{f}(u)(q-u)\du\\&=\unu{g}(c)\cdot \unu{f}(d)\cdot\frac{(q-p)^2}{2},
\end{flalign*}
for some $c,d\in(p,q)$. Because
\[\abs{\unu{g}(c)\cdot \unu{f}(d)-\unu{f}(\xi)\unu{g}(\xi)}\leq \norm{\unu{g}}\cdot\omega(\unu{f},q-p)+\norm{\unu{f}}\cdot\omega(\unu{g},q-p)\]
the proof is complete.
\end{proof}
\begin{lemma}\label{holhos-lemma-main-2}
Consider $f$ a function of class $C^1$ defined on $[a,b]$ with the derivative $\unu{f}$ having a finite number of zeros. Let $g\in C^{1}[a,b]$ be a function with $\unu{g}(t)>0$ for every $t$ in $[a,b]$. Then
\[\liminf_{n\rightarrow\infty}\ n\paran{\int_{a}^{b}f(t)\dg-RS(f,g,\Delta_{opt},\min)}\geq\frac{1}{2}\paran{\int_{a}^{b}\sqrt{|\unu{f}(t)|\cdot \unu{g}(t)}\dt}^2.\]
\end{lemma}
\begin{proof}
We first prove that for any $\delta>0$ there exists a positive integer $r$ such that for any $n$-division $\Delta$ of $[a,b]$ the following inequality is true:
\begin{flalign*}(n+r)^{\frac{1}{2}}\paran{\int_{a}^{b}f(t)\dg-RS(f,g,\Delta,\min)}^{\frac{1}{2}}\geq \frac{1}{\sqrt{2}}\int_{a}^{b}\sqrt{|\unu{f}(t)| \unu{g}(t)}\dt- \delta (b-a).
\end{flalign*}
Since the function $x\mapsto x^{1/2}$ is uniformly continuous on $[0,\infty)$, there exists $\delta_{1}>0$ such that for any $x$ and $y$ in $[0,\infty)$ if $|x-y|<\delta_{1}$ then $|x^{1/2}-y^{1/2}|<\delta$. We take a subinterval $[p,q]$ of $[a,b]$ and suppose $\unu{f}(t)\neq 0$ in $[p,q]$. Because of the continuity of the derivatives of $f$ and $g$ there exists $\eta>0$ such that if $q-p<\eta$ then $\frac{1}{2}\paranp{\norm{\unu{g}}\cdot\omega(\unu{f},q-p)+\norm{\unu{f}}\cdot\omega(\unu{g},q-p)}<\delta_{1}$. Using Lemma \ref{holhos-lemma-modulus} we obtain
\[\abs{\frac{\int_{p}^{q}f(t)\dg-[g(q)-g(p)]\min_{t\in[p,q]}f(t)}{(q-p)^2}- \frac{1}{2}|\unu{f}(\xi)|\unu{g}(\xi)}\leq \delta_{1},\] for any $\xi\in[p,q]$. Therefore, we have
\[\abs{\frac{\paran{\int_{p}^{q}f(t)\dg-[g(q)-g(p)]\min_{t\in[p,q]}f(t)}^{\frac{1}{2}}}{q-p}- \frac{1}{\sqrt{2}}\sqrt{|\unu{f}(\xi)|\unu{g}(\xi)}}\leq \delta,\] which is equivalent with
\[\abs{\paran{\int_{p}^{q}f(t)\dg-[g(q)-g(p)]\min_{t\in[p,q]}f(t)}^{\frac{1}{2}}-\frac{1}{\sqrt{2}}\sqrt{|\unu{f}(\xi)|\unu{g}(\xi)} (q-p)}\leq \delta(q-p).\]
\par Since $\unu{f}$ is uniformly continuous on $[a,b]$, for the above $\delta>0$ there exists $\zeta>0$ such that $|x-y|<\zeta$ implies $|\unu{f}(x)-\unu{f}(y)|<\delta^2/\norm{\unu{g}}$. We denote by $Z$ the zero set of $\unu{f}$:
\[Z=\set{t\in[a,b]\,|\, \unu{f}(t)=0}\] and define the $\zeta$-neighborhood $Z_{\zeta}$ of $Z$ by
\[Z_{\zeta}=\set{u\in [a,b]\, |\, \exists\, t\in Z:\ |t-u|<\zeta}.\]
Then for any $t\in Z_{\zeta}$ we have $\unu{g}(t)|\unu{f}(t)|<\delta^2$ and $\unu{f}$ is not equal to 0 on the complement of $Z_\zeta$. By the definition of $Z_\zeta$ and the properties of $\unu{f}$ we can see that $Z_\zeta$ is a disjoint union of finitely many intervals (by choosing $\zeta$ small enough). We denote by $r_{1}$ the number of all endpoints of the intervals of $Z_\zeta$. For $\eta>0$ obtained above we take a positive integer $r_{2}$ satisfying $r_{2}\geq (b-a)/\eta$ and set $r=r_{1}+r_{2}$. For any $n$-division $\Delta$ of $[a,b]$ we can add at most $r_{2}$ points to $\Delta$ such that the mesh of the new division is less than or equal to $\eta$. Moreover we add the endpoints of all the intervals of $Z_\zeta$ and denote the new division by
\[\Delta': t_{0}=a<t_{1}<\dots, <t_{m}=b.\] By the definition of $\Delta'$ we have $m\leq n+r$ and $t_i-t_{i-1}\leq \eta$. Each interval $[t_{i-1},t_i]$ satisfies $[t_{i-1},t_i]\subset \overline{Z_{\zeta}}$ or $[t_{i-1},t_i]\subset [a,b]\setminus Z_\zeta$. In both cases we can take $c_{i}\in [t_{i-1},t_i]$ satisfying
\[\int_{t_{i-1}}^{t_i}\sqrt{|\unu{f}(t)|\cdot \unu{g}(t)}\dt= \sqrt{|\unu{f}(c_i)|\cdot \unu{g}(c_i)} (t_i-t_{i-1}).\] In the case $[t_{i-1},t_i]\subset \overline{Z_{\zeta}}$ we have
\begin{flalign*}
\frac{1}{\sqrt{2}}&\sqrt{|\unu{f}(c_i)|\cdot \unu{g}(c_i)} (t_i-t_{i-1})\leq \frac{1}{\sqrt{2}}\delta (t_i-t_{i-1})\\&\leq \paran{\int_{t_i}^{t_{i-1}}f(t)\dg-[g(t_{i})-g(t_{i-1})]\min_{t\in[t_{i-1},t_i]}f(t)}^{\frac{1}{2}}+\delta(t_{i}-t_{i-1}).
\end{flalign*}
In the case $[t_{i-1},t_i]\subset [a,b]\setminus Z_\zeta$, $\unu{f}$ is not equal to 0 in $[t_{i-1},t_i]$, so
\begin{flalign*}
\frac{1}{\sqrt{2}}&\sqrt{|\unu{f}(c_i)|\cdot \unu{g}(c_i)} (t_i-t_{i-1})\\&\leq \paran{\int_{t_i}^{t_{i-1}}f(t)\dg-[g(t_{i})-g(t_{i-1})]\min_{t\in[t_{i-1},t_i]}f(t)}^{\frac{1}{2}}+\delta(t_{i}-t_{i-1}).
\end{flalign*}
Adding all these inequalities for $i=1,2,\dots,m$ we get
\begin{flalign*}
\frac{1}{\sqrt{2}}&\int_{a}^{b}\sqrt{|\unu{f}(t)|\cdot \unu{g}(t)}\dt\\&\leq \sum_{i=1}^{m}\paran{\int_{t_i}^{t_{i-1}}f(t)\dg-[g(t_{i})-g(t_{i-1})]\min_{t\in[t_{i-1},t_i]}f(t)}^{\frac{1}{2}}+\delta(b-a).
\end{flalign*}
Applying the Cauchy-Schwarz inequality to the first term of the right-hand side of the above inequality we obtain
\begin{flalign*}
\sum_{i=1}^{m}&\paran{\int_{t_i}^{t_{i-1}}f(t)\dg-[g(t_{i})-g(t_{i-1})]\min_{t\in[t_{i-1},t_i]}f(t)}^{\frac{1}{2}}\\
&\leq m^{\frac{1}{2}}\paran{\sum_{i=1}^{m}\paran{\int_{t_i}^{t_{i-1}}f(t)\dg-[g(t_{i})-g(t_{i-1})]\min_{t\in[t_{i-1},t_i]}f(t)}}^{\frac{1}{2}}\\&= m^{\frac{1}{2}}\paran{\int_{a}^{b}f(t)\dg- RS(f,g,\Delta',\min)}^{\frac{1}{2}}.
\end{flalign*}
From these we have
\begin{flalign*}
\frac{1}{\sqrt{2}}&\int_{a}^{b}\sqrt{|\unu{f}(t)|\cdot \unu{g}(t)}\dt \leq m^{\frac{1}{2}}\paran{\int_{a}^{b}f(t)\dg- RS(f,g,\Delta',\min)}^{\frac{1}{2}}+\delta(b-a).
\end{flalign*}
Because
\[[g(c)-g(b)]\min_{t\in [b,c]}f(t)+ [g(b)-g(a)]\min_{t\in [a,b]}f(t)\geq [g(c)-g(a)]\min_{t\in [a,c]}f(t),\] for every $a<b<c$, we obtain
\[\int_{a}^{b}f(t)\dg- RS(f,g,\Delta',\min)\leq \int_{a}^{b}f(t)\dg- RS(f,g,\Delta,\min).\]
This estimate and $m\leq n+r$ imply
\begin{flalign*}\frac{1}{\sqrt{2}}\int_{a}^{b}\sqrt{|\unu{f}(t)| \unu{g}(t)}\dt \leq (n+r)^{\frac{1}{2}}\paran{\int_{a}^{b}f(t)\dg-RS(f,g,\Delta,\min)}^{\frac{1}{2}}+ \delta (b-a).
\end{flalign*}
Now, let us prove Lemma \ref{holhos-lemma-main-2}. From the continuity of the function $x\mapsto x^2$ in $x_{0}=\frac{1}{\sqrt{2}}\int_{a}^{b}\sqrt{|\unu{f}(t)| \unu{g}(t)}\dt$, for any $\varepsilon>0$ there exists $\xi>0$ such that if $x_{0}-x\leq \xi$ we have $x_{0}^2-x^2\leq \frac{\varepsilon}{2}$. We take $\delta>0$ which satisfies $\delta(b-a)<\xi$. We can apply the result obtained above and get a positive integer $r$ such that for any $n$-division $\Delta$ of $[a,b]$
\begin{flalign*}\xi&\geq \delta(b-a)\\
&\geq\frac{1}{\sqrt{2}}\int_{a}^{b}\sqrt{|\unu{f}(t)| \unu{g}(t)}\dt - (n+r)^{\frac{1}{2}}\paran{\int_{a}^{b}f(t)\dg-RS(f,g,\Delta,\min)}^{\frac{1}{2}},
\end{flalign*}
which implies
\begin{flalign*}\frac{\varepsilon}{2}\geq \frac{1}{2}\paran{\int_{a}^{b}\sqrt{|\unu{f}(t)| \unu{g}(t)}\dt}^2 - (n+r)\paran{\int_{a}^{b}f(t)\dg-RS(f,g,\Delta,\min)}.
\end{flalign*}
We can substitute the optimal division $\Delta_{opt}$ for $\Delta$ in the above inequality and get
\begin{flalign*} (n+r)\paran{\int_{a}^{b}f(t)\dg-RS(f,g,\Delta_{opt},\min)}\geq  \frac{1}{2}\paran{\int_{a}^{b}\sqrt{|\unu{f}(t)| \unu{g}(t)}\dt}^2 -\frac{\varepsilon}{2}.
\end{flalign*}
Since
\[\lim_{n\rightarrow\infty}\paran{\int_{a}^{b}f(t)\dg-RS(f,g,\Delta_{opt},\min)}=0,\] we can choose a positive integer $N$ such that for $n\geq N$ the inequality
\[0\leq r\paran{\int_{a}^{b}f(t)\dg-RS(f,g,\Delta_{opt},\min)}\leq \frac{\varepsilon}{2}\] holds.
Thus
\[n\paran{\int_{a}^{b}f(t)\dg-RS(f,g,\Delta_{opt},\min)}\geq \frac{1}{2}\paran{\int_{a}^{b}\sqrt{|\unu{f}(t)| \unu{g}(t)}\dt}^2-\varepsilon,\] for every $n\geq N$. This completes the proof.
\end{proof}
\begin{theorem}\label{holhos-theorem-main}
Consider $f$ a function of class $C^1$ defined on $[a,b]$ with the derivative $\unu{f}$ having a finite number of zeros. Let $g\in C^{1}[a,b]$ be a function with $\unu{g}(t)>0$ for every $t$ in $[a,b]$. Then
\[\lim_{n\rightarrow\infty}n\paran{\int_{a}^{b}f(t)\dg-RS(f,g,\Delta_{opt},\min)}=\frac{1}{2}\paran{\int_{a}^{b}\sqrt{|\unu{f}(t)|\cdot \unu{g}(t)}\dt}^2.\]
\end{theorem}
\begin{proof}
The result follows from the inequalities of Lemma \ref{holhos-lemma-main-1} and \ref{holhos-lemma-main-2}.
\end{proof}

\end{document}